\documentclass[11pt]{amsart}
\usepackage{hyperref}
\hypersetup{
  colorlinks   = true, 
  urlcolor     = blue, 
  linkcolor    = red, 
  citecolor   = green 
}
\usepackage{enumerate}
\usepackage{tikz}

\newtheorem{thm}{Theorem}[section]
\newtheorem{lem}[thm]{Lemma}
\newtheorem{prop}[thm]{Proposition}

\theoremstyle{definition}
\newtheorem{defn}[thm]{Definition}
\newtheorem{exmp}[thm]{Example}

\theoremstyle{remark}
\newtheorem{rem}[thm]{Remark}

\numberwithin{equation}{section}

\newcommand{\GW}{\mathrm{GW}}
\newcommand{\Gr}{\mathrm{Gr}}
\newcommand{\Ham}{\mathrm{Ham}}
\newcommand{\Aut}{\mathrm{Aut}}
\newcommand{\Hom}{\mathrm{Hom}}
\newcommand{\smax}{\mathrm{smax}}
\newcommand{\univ}{\mathrm{univ}}
\newcommand{\Z}{\mathbb{Z}}
\newcommand{\Q}{\mathbb{Q}}
\newcommand{\R}{\mathbb{R}}
\newcommand{\C}{\mathbb{C}}

\providecommand{\abs}[1]{\lvert#1\rvert}

\begin{document}

\title{Symplectic capacities from Hamiltonian Circle Actions}
\author[T. Hwang]{Taekgyu Hwang}
\address{School of Mathematics, Korea Institute for Advanced Study, 85 Hoegiro, Dongdaemun-gu, Seoul, 130-722, Republic of Korea}
\email{hwangtaekkyu@kias.re.kr}

\author[D. Y. Suh]{Dong Youp Suh}
\address{Department of Mathematical Sciences, KAIST, 291 Daehak-ro, Yuseong-gu, Daejeon 305-701, Republic of Korea}
\email{dysuh@math.kaist.ac.kr}
\thanks{Dong Youp Suh is supported in part by Basic Science Research Program through the National Research Foundation of Korea(NRF) funded by the Ministry of Education(2013R1A1A2007780).}

\begin{abstract}
Let $M$ be a closed Fano symplectic manifold with a semifree Hamiltonian circle action with isolated maximum. We compute the Gromov width and the Hofer-Zehnder capacity of $M$ using a moment map.
\end{abstract}
\maketitle

\section{Introduction}
The celebrated Gromov's nonsqueezing theorem~\cite{Gr} says that there is a symplectic embedding of a ball $B^{2n}(r)$ of radius~$r$ into $B^2(R) \times \R^{2n-2}$ if and only if $r \leq R$. This motivates the notion of the \textbf{Gromov width} of a symplectic manifold $(M, \omega)$, defined by
\[w_G(M):= \sup \left\{\pi r^2 \mid B^{2n}(r) \text{ can be symplectically embedded into } M^{2n}\right\}.\]
Karshon and Tolman~\cite{KT}, and Lu~\cite{Lu1} independently, computed the Gromov width of the complex Grassmannian $\Gr(k,m)$ of $k$-planes in $\C^m$. It is equal to~$m$ if the cohomology class $[\omega]$ of the symplectic form is normalized to be the first Chern class. The idea of the proof by Karshon and Tolman is the following.

Consider the $S^1$~action on~$\C^m$ given by
\[t \cdot (z_1, \dots, z_m) = (e^{2\pi it}z_1, \dots, e^{2\pi it}z_k, z_{k+1}, \dots, z_m).\]
Here we are using a different action from that of ~\cite{KT} to induce the opposite action on $\Gr(k,m)$. Also, we consider $S^1$ as $\R/\Z$ instead of the unit circle in $\C$. This induces a semifree, Hamiltonian $S^1$~action on~$\Gr(k,m)$. The moment map $H\colon \Gr(k,m) \rightarrow \R$ attains its maximum exactly at the point $F_{\max}:= \C^k \times 0 \subset \C^m$. Choose a fixed component ~$F_{\smax}$ of the action so that there is no critical value between $H_{\smax}:= H(F_{\smax}) $ and $H_{\max}:= H(F_{\max})$. Karshon and Tolman showed that the preimage of the interval $(H_{\smax}, H_{\max}]$ is symplectomorphic to the ball
\[\Big\{x = (x_i)\in \C^{k(m-k)} \Bigm| \pi \sum_i \abs{x_i}^2 < H_{\max} - H_{\smax}\Big\}.\]
Since $[\omega]$ is $m$~times the generator of $H^2(\Gr(k,m),\Z)$ and all weights at $F_{\max}$ are~$-1$, the difference $H_{\max} - H_{\smax}$, which is the symplectic area of a gradient sphere from $F_{\smax}$ to $F_{\max}$, is equal to~$m$. See Lemma~\ref{lem:moment_map}. This shows that the Gromov width is at least~$m$. To check this lower bound is indeed the Gromov width, they used the fact by Gromov (Proposition~\ref{prop:upper_bound}) that a nonvanishing Gromov-Witten invariant with a point insertion gives an upper bound. By the result on the Gromov-Witten invariants of $\Gr(k,m)$ by Siebert and Tian~\cite{ST}, they concluded that $m$ is also an upper bound.

In this paper, we apply this idea to use Hamiltonian circle actions to compute the Gromov width to a broder class of symplectic manifolds. More precisely, we use Hamiltonian circle actions to find certain nonvanishing Gromov-Witten invariants, which estimate the Gromov width from above. Using a result by~Lu (Proposition~\ref{prop:Lu}), this method can also be used to estimate the Hofer-Zehnder capacity $c_{HZ}$(Definition~\ref{def:HZ}).

To state our main theorem we need some terminologies. An almost complex manifold $(M, J)$ is called \textbf{Fano} if any nonconstant $J$-holomorphic sphere has positive Chern number. By slight abuse of terminology, we say a Hamiltonian $S^1$-manifold $(M, \omega)$ is Fano if there is an $S^1$-invariant $\omega$-compatible almost complex structure $J$ such that $(M, J)$ is Fano. For example, any monotone symplectic manifold is Fano. An $S^1$~action is called \textbf{semifree} if it is free outside the fixed point set. We let $H_{\max}$, $H_{\smax}$ and $H_{\min}$ denote the largest, second largest and smallest critical values of the moment map~$H$, respectively.
\begin{thm}\label{thm:main}
Let $(M, \omega)$ be a closed Fano symplectic manifold with a semifree Hamiltonian $S^1$~action. The Gromov width and the Hofer-Zehnder capacity are estimated as
\begin{enumerate}[(a)]
\item\label{condition:a}
$w_G(M) \leq H_{\max} - H_{\min} \leq c_{HZ}(M)$.

\item\label{condition:b} Further if $F_{\max}$ is a point, then
\begin{align*}
	w_G(M) &= H_{\max} - H_{\smax}, \\
	c_{HZ}(M) &= H_{\max} - H_{\min}.
\end{align*}
\end{enumerate}
\end{thm}

Complex Grassmannians satisfy the assumptions of Theorem~\ref{thm:main}~(b). Also, the product of manifolds satisfying these assumptions again satisfies the assumptions. As an application, we compute the Gromov widths and the Hofer-Zehnder capacities of products of complex Grassmannians. We will see examples displaying the necessity of our assumptions in Section~\ref{sec:example}.

A related result was given by Lu~\cite{Lu2}. He proved that an upper bound for the Gromov width is given by $H_{\max} - H_{\min}$, under the assumption that the Hamiltonian $S^1$ action is semifree with isolated fixed points. Since this assumption implies Fano condition our estimate is sharper, while he also related other (pseudo) capacities. See Theorem~6.8 in \cite{Lu2}.

The proof of Theorem~\ref{thm:main} heavily uses results by McDuff and Tolman~\cite{McT}. They developed techniques to compute the Seidel representation from the data of the Hamiltonian circle action. The Seidel representation is a group homomorphism
\[S\colon \pi_1(\Ham(M, \omega)) \rightarrow QH^0(M; \Lambda)^{\times}\]
from the fundamental group of the Hamiltonian diffeomorphism group to the group of units of degree zero in the small quantum cohomology ring. Since the quantum multiplication is defined using Gromov-Witten invariants, we may compute some Gromov-Witten invariants once we know this homomorphism.  Although the Seidel representation is hard to compute in general, McDuff and Tolman obtained considerable information on it when the Hamiltonian loop is given by a Hamiltonian circle action. As an application, they gave a presentation of the quantum cohomology ring of a toric manifold. Gonzalez~\cite{Go} used their techniques to compute the quantum cohomology ring of symplectic manifolds admitting semifree circle action with isolated fixed points.

Suppose the action is semifree near the maximum fixed component $F_{\max}$. In this case McDuff and Tolman proved that the coefficient of the first term of the Seidel element is $[F_{\max}]$. See Theorem~\ref{thm:Seidel_element}. They also showed in~\cite{McT} (as a corollary of Theorem~1.10~(iii)) that all other terms vanish when $(M, \omega)$ is Fano and $\mathrm{codim}\, F_{\max} = 2$. The following theorem, which we use to prove Theorem~\ref{thm:main}, states that the codimension assumption is not necessary if the action is semifree on the whole manifold.
\begin{thm}\label{thm:main2}
Let $\phi$ be a semifree Hamiltonian $S^1$ action on a closed Fano symplectic manifold~$M$. Then
\[
	S(\phi) = [F_{\max}] \otimes q^{m(F_{\max})} t^{-H_{\max}},
\]
where $m(F_{\max})$ is the sum of weights at~$F_{\max}$.
\end{thm}

In Section~\ref{sec:Seidel}, we introduce results by McDuff and Tolman on the Seidel representation. We prove Theorem~\ref{thm:main2} in Section~\ref{sec:proof} and use it to prove Theorem~\ref{thm:main}. In Section~\ref{sec:example}, we compute the Gromov width and the Hofer-Zehnder capacity of the product of complex Grassmannians. We also give examples showing that we cannot remove our assumptions.

\subsection*{Acknowledgements}
We thank Kaoru Ono who suggested that a similar result holds for the Hofer-Zehnder capacity, Guangcun Lu for helpful comments, and the anonymous referee who strengthened Theorem~\ref{thm:main} and Theorem~\ref{thm:main2} in earlier version. We thank another referee who read the paper very carefully and gave us valuable suggestions. His or her suggestions greatly improved expositions in our paper.

\section{Hamiltonian circle action and Seidel representation}\label{sec:Seidel}
The Seidel representation was first introduced by Seidel~\cite{S} and later defined by McDuff~\cite{Mc} in the form described in this section. A Hamiltonian circle action represents an element in $\pi_1(\Ham(M,\omega))$. McDuff and Tolman~\cite{McT} gave a method to compute the Seidel representation for such elements. Their results are essential tools to the proof of our theorems.

\subsection{Hamiltonian circle action}
Let $(M, \omega)$ be a closed symplectic manifold. An $S^1$ action on $(M, \omega)$ is called \textbf{Hamiltonian} if there exists a map
\[H\colon M \rightarrow \R\]
satisfying $\omega(X, \cdot) = -dH$, where $X$ is the fundamental vector field of the action. Here we use the convention that the vector field $X$ generates the flow $\phi_t$ of period one. Such a map $H$ is called a \textbf{moment map}. The critical point set of $H$ is exactly the fixed point set of the action, which is a finite disjoint union of symplectic submanifolds. The moment map $H$ is a Morse-Bott function with critical points of even index, so the level sets of $H$ are connected. We denote by $F_{\max}$ the unique maximum fixed component with respect to $H$.

Choose an $S^1$-invariant $\omega$-compatible almost complex structure $J$, so that $g(\cdot,\cdot):= \omega(\cdot, J\cdot)$ is an $S^1$-invariant Riemannian metric. Then $-JX$ is the upward gradient vector field of $H$ with respect to $g$. Since $g$ is $S^1$-invariant, the flow $\gamma_s$ of $-JX$ commutes with $\phi_t$. For any point $w \in M \setminus M^{S^1}$, choose the gradient flow through~$w$. By rotating it by the circle action, we have a $J$-holomorphic curve
\[u\colon \R \times S^1 \rightarrow M: (s,t) \mapsto \phi_t(\gamma_s(w))\]
with the product complex structure on $\R \times S^1$. The closure of the image of~$u$ is called a \textbf{gradient sphere}.
  
The action gives weights of the tangential representation at each fixed component. Weights at $F_{\max}$ are nonpositive with our conventions. A Hamiltonian $S^1$ action is semifree if and only if all weights are one of $-1, 0, 1$. To see this, let $w \in M$ be a point with isotropy subgroup $\Z/k$. Since $\gamma_s$ commutes with the $S^1$ action, the isotropy subgroup of $\gamma_s(w)$ is also $\Z/k$ for all~$s$. The limit point as $s \rightarrow -\infty$ is a fixed point of the $S^1$ action having a multiple of~$k$ as a weight.

The moment map $H$ is determined up to constant addition. For monotone symplectic manifolds, after normalization $[\omega] = c_1(M)$, it is convenient to use the following lemma to fix a choice of~$H$ so that the moment map image of a fixed component is the negative of the sum of weights.

\begin{lem}[\cite{McT} Lemma~3.9]\label{lem:moment_map}
Let $S^1$ act on $(M, \omega)$ with moment map~$H$. Let $B$ be the homology class represented by a gradient sphere joining two fixed points $x$ and~$y$. Here the orientation of $B$ is given by a gradient flow from~$x$ to~$y$, followed by the $S^1$ action. Then
\[c_1(B) = (m(x) - m(y))/k \quad \text{and} \quad \omega(B) = (H(y) - H(x))/k,\]
where $m(x)$ is the sum of weights at fixed point $x$, and $k$ is the isotropy weight of the gradient sphere.
\end{lem}

\subsection{Seidel representation}\label{subsec:Seidel}
We define the Seidel representation for Hamiltonian circle actions. We refer to \cite{McS2, McT} for detailed explanation of the Seidel representation.

Let $\phi$ be a Hamiltonian $S^1$ action on $(M^{2n}, \omega)$ with moment map $H$. Recall that we are using the convention $S^1 = \R/\Z$ so that $\phi$ has period~$1$. This action defines a Hamiltonian bundle $\pi\colon P_{\phi} \rightarrow S^2$ with fiber $M$ as follows. Let $S^3$ be the unit sphere in $\C^2$. Consider the following $S^1$ action
\[t \cdot (z_1, z_2, w) = \left(e^{2\pi it}z_1, e^{2\pi it}z_2, \phi_t(w)\right)\]
on the product $S^3 \times M$. The bundle is defined to be the quotient space
\[P_{\phi} := S^3 \times_{S^1} M\]
with $\pi$ induced from the projection to the first factor. $P_{\phi}$ is determined by the homotopy class of $\phi$ as an element of $\pi_1(\Ham(M, \omega))$ and has a symplectic structure extending the fiberwise symplectic form.

Two cohomology classes in $H^2(P_{\phi},\Z)$ are associated to this bundle. The first one is $c_{\mathrm{vert}}$, the first Chern class of the vertical tangent bundle of $P_{\phi}$. The second one is the coupling class $u_{\phi}$, the unique class satisfying $u_{\phi}|_M = [\omega]$ and $u_{\phi}^{n+1} = 0$. We use
\begin{equation}\label{eq:Novikov}
\Lambda:= \Lambda^{\univ}[q,q^{-1}]
\end{equation}
as the coefficient ring for quantum cohomology where $q$ is a variable of degree~$2$ and
\[\Lambda^{\univ}:= \Big\{\sum_{\kappa \in \R}
r_{\kappa}t^{\kappa} \;\Big|\; r_{\kappa} \in \Q,\;
\#\{\kappa < c \mid r_{\kappa} \neq 0\}<\infty, \forall c \in \R\Big\}\]
with $\deg t =0$.

\begin{defn}\label{def:Seidel_element}
The \textbf{Seidel element} is defined to be
\[S(\phi)= \sum_{\sigma \in H_2^{\sec}(P_{\phi})}
b_{\sigma} \otimes q^{c_{\mathrm{vert}}(\sigma)} t^{u_{\phi}(\sigma)}
\in QH^0(M; \Lambda)\]
where $b_{\sigma} \in H^*(M, \Q)$ is a unique class satisfying
\begin{equation}\label{eq:Seidel_element}
\int_M b_{\sigma} \cup c = \GW_{\sigma, 1}^{P_{\phi}}(i_*c)
\end{equation}
for all $c \in H^*(M, \Q)$.
Here, $i_*\colon H^*(M, \Q) \rightarrow H^{*+2}(P_{\phi}, \Q)$ is the cohomology pushforward of inclusion $i\colon M \hookrightarrow P_{\phi}$ defined using Poincar\'e duality. The section homology $H_2^{\sec}(P_{\phi})$ consists of $\sigma \in H_2(P_{\phi}, \Z)$ with $\pi_* \sigma = [S^2]$.
\end{defn}

\begin{rem}
McDuff and Tolman~\cite{McT} used quantum homology to explain their results. Using Poincar\'e duality
\[QH_*(M;\check{\Lambda}) \rightarrow QH^*(M;\Lambda):
a \otimes q^d t^{\kappa} \mapsto PD(a) \otimes q^{-d} t^{-\kappa},\]
we present their results in the cohomology form. Here $\check{\Lambda}$ is the homology version of the Novikov ring.
\end{rem}

McDuff~\cite{Mc} proved that $S\colon \pi_1(\Ham(M, \omega)) \rightarrow QH^0(M; \Lambda)^{\times}$ is a group homomorphism. The multiplication by $S(\phi)$ gives an automorphism of the quantum cohomology ring.

\begin{defn}
The \textbf{Seidel representation} is the group homomorphism 
\[S\colon \pi_1(\Ham(M, \omega)) \rightarrow \Aut(QH^*(M; \Lambda))\]
defined by
\begin{equation}
S(\phi)(a) = S(\phi) * a
\end{equation}
for $a \in QH^*(M; \Lambda)$ where $*$ denotes the small quantum cup product.
\end{defn}

It was shown in~\cite{Mc} that for $a \in H^*(M, \Q)$,
\[S(\phi)(a) = \sum_{\sigma \in H_2^{\sec}(P_{\phi})}
b_{\sigma, a}\otimes q^{c_{\mathrm{vert}}(\sigma)} t^{u_{\phi}(\sigma)}\]
where $b_{\sigma, a} \in H^*(M, \Q)$ is a unique class satisfying
\begin{equation}\label{eq:Seidel_representation}
\int_M b_{\sigma, a} \cup c = \GW_{\sigma, 2}^{P_{\phi}}
(i_*a, i_*c)
\end{equation}
for all $c \in H^*(M, \Q)$. Here $i_*$ is the cohomology pushforward as in Definition~\ref{def:Seidel_element}. The class $b_{\sigma}$ defined in~\eqref{eq:Seidel_element} is equal to $b_{\sigma, 1}$.

\subsection{Computation results by McDuff and Tolman}
We present here some of the computation results by McDuff and Tolman~\cite{McT}.

Let $F$ be a fixed component and choose $x \in F$. Let $\sigma_x \in H_2^{\sec}(P_{\phi})$ be the homology class represented by the sphere $S^3 \times_{S^1} \{x\}$. The homology class $\sigma_x$ does not depend on the choice of $x$ and will be denoted by $\sigma_F$.  McDuff and Tolman~ (Lemma~2.2 in~\cite{McT}) showed that $c_{\mathrm{vert}}(\sigma_F) = m(F)$ and $u_{\phi}(\sigma_F) = -H(F)$ where $m(F)$ is the sum of weights at $F$, and the moment map $H$ is normalized so that $\int_{M^{2n}} H\omega^n = 0$. Since we will not use the fact that $S$ is a group homomorphism, $H$ need not be normalized for our purposes. Any section class $\sigma$ is written as the sum $\sigma_F + B$ for some $B \in H_2(M, \Z)$. Using the notation
\begin{equation}
a_B:= b_{\sigma_{F_{\max}}+B} \quad \text{(defined in~\eqref{eq:Seidel_element})}
\end{equation}
the Seidel element is written as
\begin{equation}
S(\phi) =\sum_{B \in H_2(M, \Z)}
a_B\otimes q^{m(F_{\max}) + c_1(B)} t^{-H_{\max} + \omega(B)} \in QH^0(M; \Lambda).
\end{equation}
The following theorem computes the first term of the Seidel element.

\begin{thm}[\cite{McT} Theorem~1.10 (i)]\label{thm:Seidel_element}
Let $\phi$ be a Hamiltonian $S^1$ action on a closed symplectic manifold $(M, \omega)$. Suppose all negative weights at $F_{\max}$ are $-1$. Then
\[S(\phi) = [F_{\max}] \otimes q^{m(F_{\max})} t^{-H_{\max}}
+\sum_{\substack{B \in H_2(M, \Z) \\ \omega(B)>0}}
a_B\otimes q^{m(F_{\max}) + c_1(B)} t^{-H_{\max} + \omega(B)},\]
where $[F_{\max}]$ is the cohomology class represented by $F_{\max}$.
\end{thm}

The next lemma is used to get information on the latter terms. We will use this lemma later when we want to show $a_B=0$. A pseudocycle $f\colon V \rightarrow M$ is called $S^1$-invariant if $\overline{f(V)}$ is $S^1$-invariant.
\begin{lem}[\cite{McT} Lemma 3.10]\label{lem:3.10}
Let $f\colon V \rightarrow M$ be an $S^1$-invariant pseudocycle representing the cohomology class $a$. Then $\int_M a_B \cup a = 0$ if $B$ cannot be represented by an $S^1$-invariant $J$-holomorphic stable map that intersects both $F_{\max}$ and $\overline{f(V)}$.
\end{lem}

They also computed the first term of the Seidel representation. To explain their result, we need the notion of \textbf{canonical classes}. 

\begin{lem}[\cite{McT} Lemma~1.13]
Suppose $S^1$ acts on $(M, \omega)$ with a moment map $H$. For a fixed component $F$ of the action, choose any cohomology class $c_F \in H^i(F, \Q)$. Let $\alpha$ be the index of $F$ and $e_F^-$ be the equivariant Euler class of the negative normal bundle of $F$. Then there exists a unique equivariant cohomology class $e(c_F^+) \in H_{S^1}^{i+\alpha}(M, \Q)$ satisfying

\begin{enumerate}[(a)]
	\item
		$\left.e(c_F^+)\right|_{F'} = 0$ for all fixed components $F'$ with $H(F') < H(F)$,
	\item
		$\left.e(c_F^+)\right|_F = c_F \cup e_F^-$, and
	\item
		the polynomial degree of $\left.e(c_F^+)\right|_{F'}$ is less than the index of $F'$ for all fixed components $F' \neq F$.
\end{enumerate}
\end{lem}

Such a unique class $e(c_{F}^+)$ is called the canonical class. The \textbf{upward extension} $c_F^+ \in H^{i+\alpha}(M, \Q)$ is defined to be the restriction of $e(c_F^+)$ to the ordinary cohomology. The \textbf{downward extension} $c_F^- \in H^{i+2n-\alpha-\dim F}(M, \Q)$ is defined in the same way using $-H$ instead of $H$.

\begin{rem}\label{rem:basis}
 By choosing a basis of $H^*(F, \Q)$ over all fixed components $F$, we obtain a basis by canonical classes for $H_{S^1}^*(M, \Q)$ as an $H^*(BS^1, \Q)$-module. By restricting to the ordinary cohomology, the upward extensions $\{c_F^+\}$ and the downward extensions $\{c_F^-\}$ form a basis for $H^*(M, \Q)$, respectively.
\end{rem}

\begin{thm}[\cite{McT} Theorem~1.15]\label{thm:Seidel_representation}
Let $\phi$ be a semifree Hamiltonian $S^1$ action on a closed symplectic manifold $(M, \omega)$ with moment map $H$. Choose a cohomology class $c_F \in H^*(F, \Q)$ for a given fixed component $F$. Then
\[S(\phi)(c_F^-) = c_F^+ \otimes q^{m(F)} t^{-H(F)}
+\sum_{\substack{B \in H_2(M, \Z) \\ \omega(B)>0}}
a_{B,c_F}\otimes q^{m(F) + c_1(B)} t^{-H(F) + \omega(B)}.\]
Here we are using the notation $a_{B,c_F} := b_{\sigma_F + B, c_F^-}$ (defined in~\eqref{eq:Seidel_representation}) so that
\[\int_M a_{B, c_F} \cup c = \GW_{\sigma_F + B, 2}^{P_{\phi}}(i_*c_F^-, i_*c)\]
for all $c \in H^*(M, \Q)$.
\end{thm}

When the action is semifree, the upward extensions and the downward extensions have a nice geometric interpretation. Given a Riemannian metric on $M$, denote by $\psi_t$ the negative gradient flow of $H$. For $C \subset M$, the \textbf{unstable manifold} of $C$ is defined to be
\[W^u(C):= \{x \in M \mid \lim_{t \rightarrow -\infty}\psi_t(x) \in C\}.\]
\begin{prop}[\cite{McT} Proposition~4.8]\label{prop:downward extension}
Let $S^1$ act semifreely on $(M, \omega)$ with moment map $H$. Choose a Riemannian metric on $M$ associated to a generic $S^1$-invariant $\omega$-compatible almost complex structure. Then for a generic submanifold $C$ of a fixed component $F$, the unstable manifold $W^u(C)$ is an $S^1$-invariant pseudocycle. The cohomology class represented by $W^u(C)$ is equal to the downward extension $[C]^-$, where $[C]$ denotes the cohomology class in~$H^*(F)$ represented by~$C$.
\end{prop}

\section{Proof of theorems}\label{sec:proof}
The following is a restatement of Theorem~\ref{thm:main2}.

\begin{thm}\label{thm:Seidel}
Let $\phi$ be a semifree Hamiltonian $S^1$ action on a closed Fano symplectic manifold $(M, \omega)$. Write the Seidel element as in Theorem~\ref{thm:Seidel_element}.
\begin{equation*}\begin{split}
S(\phi) &= [F_{\max}]\otimes q^{m(F_{\max})} t^{-H_{\max}}\\
&\quad + \sum_{\substack{B \in H_2(M, \Z) \\ \omega(B)>0}}
a_B\otimes q^{m(F_{\max}) + c_1(B)} t^{-H_{\max} + \omega(B)} \in QH^0(M; \Lambda).
\end{split}\end{equation*}
Then $a_B = 0$ for all such $B$.
\end{thm}

\begin{proof}
Given a fixed component~$F$, we will show $\int_M a_B \cup c_F^- = 0$ for all $c_F \in H^*(F, \Q)$. This will complete the proof by Remark~\ref{rem:basis}.

Fix $F$ and $c_F \in H^*(F, \Q)$, and suppose $\int_M a_B \cup c_F^- \neq 0$.  We may assume that $c_F$ is represented by a smooth submanifold. Let $C$ be a generic submanifold of~ $F$ representing~$c_F$. By Lemma~\ref{lem:3.10} and Proposition~\ref{prop:downward extension}, there is an $S^1$-invariant $J$-holomorphic stable map in class~$B$ that intersects both $F_{\max}$ and $\overline{W^u(C)}$. It is a connected union of $S^1$-invariant $J$-holomorphic spheres. As explained in \cite{McT}(Section~3.2), the image of each such sphere is either contained in the fixed point set or a gradient sphere joining two fixed points. Therefore, there is a chain of spheres joining two fixed points $x \in \overline{W^u(C)}$ and $y \in F_{\max}$. Write $B = \sum k_i B_i + \sum l_j C_j$, where $B_i$ is the homology class of a sphere in the chain, $C_j$ is the homology class of a sphere not in the chain, and $k_i$,~$l_j \geq 1$ are multiplicities. 

Recall that the index of~$F$ is twice the number of negative weights at~$F$. Since the action is semifree,
\begin{equation}\label{eq:deg_cf}
\deg c_F^- = \dim M - \dim W^u(C) = \dim M - \mathrm{index}(F) \leq \dim M + 2m(F).
\end{equation}
On the other hand, by Lemma~\ref{lem:moment_map} and the assumption that $(M,J)$ is Fano,
\begin{equation}\label{eq:c1}
c_1(B) \geq \sum c_1(B_i) \geq m(x) - m(y) \geq m(F) - m(F_{\max}).
\end{equation}
Since the Seidel element~$S(\phi)$ has degree~$0$, the inequality~\eqref{eq:c1} implies
\begin{equation}\label{eq:deg_ab}
\deg a_B = -2m(F_{\max})-2c_1(B) \leq -2m(F).
\end{equation}
We have assumed $\int_M a_B \cup c_F^- \neq 0$, hence, $\deg a_B + \deg c_F^- = \dim M$. This implies that inequalities in~\eqref{eq:deg_cf} and~\eqref{eq:deg_ab} are equalities. This is possible only if $F = F_{\max}$ and $c_1(B)=0$, which contradicts the Fano assumption.
\end{proof}

Before proving Theorem~\ref{thm:main}, we recall the definition of the Hofer-Zehnder capacity. See \cite{McS1} Chapter~12 for details. We assume for simplicity that $(M, \omega)$ is closed. Let $X_K$ denote the Hamiltonian vector field of a function $K \colon M \rightarrow \R$. A function $K\colon M \rightarrow \R$ is called admissible if
\begin{enumerate}[(a)]
\item there exist open sets $U$ and~$V$ such that $K|_U = K_{\min}$ and $K|_V = K_{\max}$,
\item $X_K$ has no nonconstant periodic orbit of period less than~$1$.
\end{enumerate}
\begin{defn}\label{def:HZ}
The Hofer-Zehnder capacity is defined to be
\[c_{HZ}(M):= \sup\left\{K_{\max} - K_{\min} \mid \textrm{$K$ is admissible}\right\}.\]
\end{defn}
In the proof of Theorem~\ref{thm:main} we use the following four propositions. The last proposition is the key one used in our proof.

\begin{prop}[\cite{KT} Proposition~2.8]\label{prop:KT}
Let $(M, \omega)$ be a closed symplectic manifold with a Hamiltonian $S^1$ action. If all weights at~$F_{\max}$ are $-1$, then
\[
	w_G(M) \geq H_{\max} - H_{\smax}.
\]
\end{prop}

\begin{prop}[Gromov]\label{prop:upper_bound}
Suppose $\GW_{A, k}^M\left([pt], \alpha_2, \dots, \alpha_k\right) \neq 0$ for $k \geq 1$ and $A \neq 0$. Then $w_G(M) \leq \omega(A)$.
\end{prop}
\begin{proof}
See Proposition~4.1 in \cite{KT} or Section~1.3 in \cite{McS2}. The proof applies to general $A \in H_2(M, \Z)$ using Gromov compactness (\cite{McS2} Theorem~5.3.1). If there is a $J$-holomorphic stable map in class~$A$ passing through a given point, then there is a $J$-holomorphic sphere in some class $A_i$ passing through a given point with $0 < \omega(A_i) \leq \omega(A)$.
\end{proof}

\begin{prop}[\cite{Lu1} Corollary~1.19]\label{prop:Lu}
If $\GW_{A,k}^M([pt], [pt], \alpha_3, \dots, \alpha_k) \neq 0$ for $k \geq 2$, then $c_{HZ}(M) \leq \omega(A)$.
\end{prop}

\begin{prop}\label{prop:key}
Under the same assumption as Theorem~\ref{thm:main}, given a fixed component~$F \neq F_{\max}$ and nonzero $c_F \in H^*(F, \Q)$, there exist $A \in H_2(M, \Z)$ and $\alpha \in H^*(M, \Q)$ such that
\[
\GW_{A,3}^M ([F_{\max}], c_F^-, \alpha) \neq 0 \quad\text{and} \quad \omega(A) = H_{\max} - H(F).
\]
\end{prop}
\begin{proof}
By Theorem~\ref{thm:Seidel_representation},
\begin{equation}\begin{split}
S(\phi)*c_F^- &= c_F^+ \otimes q^{m(F)} t^{-H(F)}\\
&\quad + \sum_{\substack{B \in H_2(M, \Z) \\ \omega(B)>0}}
a_{B,c_F} \otimes q^{m(F)+c_1(B)} t^{-H(F)+\omega(B)}.
\end{split}\end{equation}
We use Theorem~\ref{thm:Seidel} and compare the coefficients of $q^{m(F)} t^{-H(F)}$. Since only the first term in~$S(\phi)$ survives, we have
\begin{equation}
c_F^+ = \sum_{\substack{B \in H_2(M, \Z) \\ c_1(B) = m(F) - m(F_{\max}) \\ \omega(B) = H_{\max} - H(F)}}([F_{\max}]*c_F^-)_B,
\end{equation}
where $(a*b)_B$ denotes a unique class satisfying $\int_M (a*b)_B \cup c = \GW_{B,3}^M(a,b,c)$ for all $c \in H^*(M,\Q)$. Since $c_F^+ \neq 0$, there exists $A \in H_2(M, \Z)$ such that
\begin{equation}
\omega(A) = H_{\max} - H(F) \quad \text{and} \quad ([F_{\max}]*c_F^-)_A \neq 0.
\end{equation}
Now choose $\alpha \in H^*(M, \Q)$ so that $\int_M ([F_{\max}]*c_F^-)_A \cup \alpha \neq 0$. Then
\begin{equation}
\GW_{A,3}^M([F_{\max}], c_F^-, \alpha) = \int_M ([F_{\max}]*c_F^-)_A \cup \alpha \neq 0.
\end{equation}
\end{proof}

\begin{proof}[\textbf{Proof of Theorem~\ref{thm:main}}]\leavevmode
\begin{enumerate}[(a)]
\item
To prove the first inequality, we apply Proposition~\ref{prop:key} when $F = F_{\min}$ and $c_F = [pt]$. Since $c_F^- = [pt]$, we use Proposition~\ref{prop:upper_bound} to obtain
\[
	w_G(M) \leq \omega(A) = H_{\max} - H_{\min}.
\]
The second inequality follows by slightly changing our moment map~$H$ to an admissible function~$K$. For any small $\epsilon>0$, we find admissible~$K$ with $K_{\min} = H_{\min} + \pi\epsilon^2$ and $K_{\max} = H_{\max} - \pi\epsilon^2$ as follows. Take a Darboux-Weinstein neighborhood~$U$ of~$F_{\min}$. In this neighborhood the moment map is written as $H = \pi \abs{z}^2 + H_{\min}$, where $\abs{z}$ is the distance from $F_{\min}$ along the normal direction. Choose a smooth function $f\colon [0,2\epsilon] \rightarrow \R$ satisfying
\[0 \leq f'(r) \leq (r^2)', \quad f(r) = \epsilon^2 \text{ near } 0, \quad f(r)=r^2 \text{ near }  2\epsilon.\]
We define $K$ on~$U$ to be $\pi f(\abs{z}) + H_{\min}$. Then nonconstant orbits of $X_K$ have period at least~$1$ by the condition on the derivative of~$f$. We define $K$ similarly on a neighborhood~$V$ of $F_{\max}$ and extend it to the whole manifold by assigning $K=H$ in the complement of $U$ and~$V$. This proves $c_{HZ}(M) \geq H_{\max} - H_{\min}$.
\item
Proposition~\ref{prop:KT} shows $w_G(M) \geq H_{\max} - H_{\smax}$. To prove this is an equality, we apply Proposition~\ref{prop:key} when $F=F_{\smax}$ and $c_F \neq 0$. Since $F_{\max}$ is a point, we use Proposition~\ref{prop:upper_bound} to have $w_G(M) \leq H_{\max} - H_{\smax}$.

The second equality follows if we apply Proposition~\ref{prop:key} when $F=F_{\min}$ and $c_F = [pt]$. This time $[F_{\max}] = c_F^- =[pt]$, so we use Proposition~\ref{prop:Lu} to get $c_{HZ}(M) \leq H_{\max} - H_{\min}$.
\end{enumerate}
\end{proof}

\section{Examples}\label{sec:example}
In this section, we use Theorem~\ref{thm:main} to compute the Gromov width and the Hofer-Zehnder capacity of the product of complex Grassmannians. We give two examples which show that our assumptions cannot be removed.
\begin{exmp}
Let $\Gr(k,m)$ be the Grassmannian of $k$-planes in $\C^m$. We assume $k \leq m-k$ for simplicity. This is a monotone symplectic manifold. The $S^1$ action on $\C^m$ given by
\[t \cdot (z_1, \dots, z_m) = (e^{2\pi it}z_1, \dots, e^{2\pi it}z_k, z_{k+1}, \dots, z_m)\]
induces a Hamiltonian $S^1$ action on $\Gr(k,m)$. To find fixed components, decompose $\C^m$ as the sum of $V_1:= \C^k \times 0$ and $V_2:= 0 \times \C^{m-k}$. If $W$ is an invariant subspace and $v=v_1 + v_2 \in V_1 \oplus V_2$ is an element of $W$, then both $v_1$ and $v_2$ are elements of~$W$. So $W$ is a product of subspaces of $V_1$ and $V_2$. We see that fixed point components are
\[\Gr(k_1, k) \times \Gr(k_2, m-k)\]
for nonnegative integers $k_1+k_2=k$. To compute the weights, decompose the tangent space of $\Gr(k,m)$ as
\begin{align*}
\Hom(\C^k, \C^{m-k}) &= \Hom(\C^k\cap V_1 \oplus \C^k\cap V_2,\, \C^{m-k}\cap V_1 \oplus \C^{m-k}\cap V_2)\\
&= \Hom(\C^{k_1}, \C^{k-k_1})\oplus
\Hom(\C^{k_1}, \C^{m-k-k_2})\\
&\quad \oplus \Hom(\C^{k_2}, \C^{k-k_1})\oplus
\Hom(\C^{k_2}, \C^{m-k-k_2}).
\end{align*}
Weights on $V_1$ are~$1$ and weights on $V_2$ are~$0$. Hence, the first and the fourth factor give weights~$0$, the second factor gives weights~$-1$, and the third factor gives weights~$1$. This shows the action is semifree. The moment map image of $\Gr(k_1,k) \times \Gr(k_2,m-k)$, the negative of the sum of weights by~Lemma~\ref{lem:moment_map}, is
\[k_1(m-k-k_2) - k_2(k-k_1) = k_1(m-k) - k_2k.\]
The largest value is $k(m-k)$ when $k_1=k$, the second largest value is $k(m-k)-m$ when $k_1=k-1$, and the smallest value is $-k^2$ when $k_1 = 0$. By Theorem~\ref{thm:main}, we have
\begin{equation}\label{ex:gr}
w_G(\Gr(k,m)) = m, \quad c_{HZ}(\Gr(k,m)) = km.
\end{equation}
\end{exmp}

In the following example, we see that the Gromov width of the product of symplectic manifolds satisfying the assumptions in Theorem~\ref{thm:main}~(b) is given by the minimum of the width of the factors. It can be larger in general. For example, Lalonde~\cite{La} proved the following: Let $S^2(2)$ be the sphere of Gromov width~$2$ and let $T^2(1)$ be the torus of Gromov width~$1$. The product $S^2(2) \times T^2(1)$ has Gromov width~$2$. For the Hofer-Zehnder capacity, it is given by the sum of the capacities of the factors.
\begin{exmp}
For $i=1, \dots, k$, let $(M_i, \omega_i)$ be a symplectic manifold satisfying the assumptions of Theorem~\ref{thm:main}~(b), so that $(M,J_i)$ is Fano. Consider the product $(M, \omega) := (M_1, \omega_1) \times \dots \times (M_k, \omega_k)$ with the compatible almost complex structure $J:= J_1 \times \dots \times J_k$. Let $A\in H_2(M, \Z)$ represent a $J$-holomorphic curve. Projections $\pi_i\colon M \rightarrow M_i$ are holomorphic, so ${\pi_i}_*A$ represents a $J_i$-holomorphic curve. Then
\[\langle c_1(M), A\rangle = \sum_i \langle \pi_i^* c_1(M_i), A\rangle = \sum_i \langle c_1(M_i), {\pi_i}_* A\rangle \geq 0.\] 
If this is zero, all ${\pi_i}_*A$ represent constant curves, which means $A=0$. So $(M, J)$ is Fano. When given the diagonal action, it satisfies the assumptions of Theorem~\ref{thm:main}~(b). The induced moment map~$H$ is given by $H_1 + \dots + H_k$, where $H_i$ is the moment map for each factor. Since the critical point set of~$H$ coincides with the fixed point set, Theorem~\ref{thm:main} implies that
\begin{equation}
w_G(M, \omega) = \min_i \left\{w_G(M_i, \omega_i)\right\}
\end{equation}
and
\begin{equation}
c_{HZ}(M, \omega) = \sum_i c_{HZ}(M_i, \omega_i).
\end{equation}
\end{exmp}

The semifree assumption in our theorem cannot be removed.
\begin{exmp}\label{ex:blowup}
Let $M$ be a symplectic manifold obtained by blowing up the toric manifold $\mathbb{P}^2$ at a fixed point. Choose the blow up size so that the resulting manifold is monotone. $M$ is a toric manifold with the moment map image of the $T^2$ action given in Figure~\ref{fig:moment map image}. We denote the divisor corresponding to each facet by $D_i$ and their intersection point by $p_{ij} = D_i \cap D_j$. The weights of $T^2$ action at each fixed point $p_{ij}$ are given by outward primitive vectors along the edges.

A pair of relatively prime integers $(a,b)$ gives a subcircle action on $M$, induced from the inclusion $S^1 \hookrightarrow S^1 \times S^1$ given by $t \mapsto (t^a, t^b)$. The induced weights of this circle action is given by the inner product of weights of $T^2$ action with the vector $(a,b)$. Let $\phi_1$ be the circle action given by $(0,1)$. This action is semifree, $F_{\max} = p_{23}$, $F_{\smax} = p_{34}$, and $F_{\min} = D_1$. The moment map image is the negative of the sum of weights, so $H_1(F_{\max}) = 2$, $H_1(F_{\smax}) = 0$, and $H_1(F_{\min}) = -1$. By Theorem~\ref{thm:main}, we see that $w_G(M) = 2$ and $c_{HZ}(M) = 3$. Now consider the circle action $\phi_2$ given by $(-1,-2)$. Then $F_{\max} = p_{14}$, $F_{\smax} = p_{34}$, and $F_{\min} = p_{23}$. The moment map image is given by $H_2(F_{\max}) = 2$, $H_2(F_{\smax}) = 1$, and $H_2(F_{\min}) = -3$. In this case the action is not semifree; $D_3$ has isotropy weight~$2$.
\end{exmp}
\begin{figure}
\caption{Moment map image of $M = \mathbb{P}^2\#\overline{\mathbb{P}^2}$ in Example~\ref{ex:blowup}.}
\label{fig:moment map image}
\begin{tikzpicture}[scale=1.5]
	\draw[thick] (1,0) -- node[below]{$D_1$} (3,0) -- node[above right]{$D_2$} (0,3)
	-- node[left]{$D_3$} (0,1) -- node[below left]{$D_4$} (1,0);
	\node[left] at (0,1) {$(0,1)$};
	\node[left] at (0,3) {$(0,3)$};
	\node[below] at (1,0) {$(1,0)$};
	\node[below] at (3,0) {$(3,0)$};

	\node[right] at (3,2) {$\phi_1 : (0,1)$};
	\node[right] at (3,1.5) {$\phi_2 : (-1,-2)$};
\end{tikzpicture}
\end{figure}
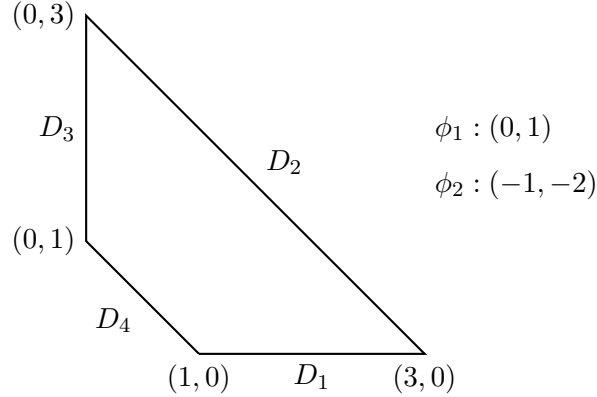

The assumption that $F_{\max}$ is a point is also essential in part~(b) of our  theorem.
\begin{exmp}
Pick a volume form~$\sigma$ on $S^2$ so that $\int_{S^2} \sigma = 1$. Consider the product $(S^2 \times S^2, a\sigma + b\sigma)$ with $a \geq b > 0$. We apply Theorem~\ref{thm:main} using the diagonal circle action to have $w_G(M) = b$ and $c_{HZ}(M) = a+b$.

On the other hand, consider the circle action rotating the first factor. There are exactly two fixed components $F_{\max} = \{N\} \times S^2$ and $F_{\min} = \{S\} \times S^2$ where $N$ and $S$ denote the north and south pole, respectively. The difference $H_{\max} - H_{\min} = H_{\max} - H_{\smax}$ is $a$. We see $b \leq a \leq a+b$, but inequalities are strict unless $a=b$.
\end{exmp}

\end{document}